\documentclass[11pt]{amsart}  
\usepackage{amsmath,amssymb,amsthm,amsfonts}
\usepackage{mathrsfs}
\usepackage{xcolor}
\usepackage[all,cmtip]{xy}

\usepackage{geometry}
\newtheorem{theorem}{Theorem}[section]
\newtheorem{proposition}[theorem]{Proposition}

\newtheorem{lemma}[theorem]{Lemma}

\newtheorem{corollary}[theorem]{Corollary}

\newtheorem*{question*}{Question}

\theoremstyle{definition}
\newtheorem{definition}[theorem]{Definition}

\newtheorem*{definition*}{Definition}

\theoremstyle{remark}

\newcommand{\ux}{\underline{x}}
\newcommand{\uh}{\underline{h}}

\newcommand{\EE}{\mathbb{E}}
\newcommand{\CC}{\mathbb{C}}

\newcommand{\QQ}{\mathbb{Q}}

\newcommand{\NN}{\mathbb{N}}
\newcommand{\FF}{\mathbb{F}}

\newcommand{\ZZ}{\mathbb{Z}}

\newcommand{\cP}{\mathcal{P}}

\newcommand{\cL}{\mathcal{L}}

\title{On weighted multilinear polynomial averages in finite fields}
\author{Guo-Dong Hong}
\address{Department of Mathematics, California Institute of Technology, Pasadena, CA 91125, USA}
\email{ghong@caltech.edu}

\begin{document}
\maketitle

\begin{abstract}
    We study the weighted multilinear polynomial averages in finite fields.
   The essential ingredient is the $u^s$-norm control of the corresponding weighted multilinear polynomial averages in finite fields, which is motivated by Teräväinen \cite{T24}. As an application, we prove an asymptotic formula for the number of the following multidimensional rational function progressions in the subsets of $\FF_p^D$:
   \[
\ux, \ux+ P_1(\varphi(y))v_1,\cdots, \ux+ P_k(\varphi(y))v_k,
\]
where $\mathcal{V}=\{v_1, \cdots, v_{k} \in \ZZ^D\}$ is a collection of nonzero vectors, $\mathcal{P}= \{P_1, \cdots, P_{k}\in \ZZ[y]\}$ is a collection of linearly independent polynomials with zero constant terms, and $\varphi(y) \in \QQ(y)$ is a nonzero rational function.

\end{abstract}

\section{Introduction}

\subsection{Weighted multilinear polynomial averages}

Throughout the paper, let $D,k \in \NN$, $\mathcal{V}=\{v_1, \cdots, v_{k} \in \ZZ^D\}$ be a collection of nonzero vectors, and $\mathcal{P}= \{P_1, \cdots, P_{k}\in \ZZ[y]\}$ be a collection of linearly independent polynomials with zero constant terms. We set $d=d(\cP):=\max_{i=1,\cdots,k} \text{deg}(P_i)$.

\begin{definition}
    We define the $u^s$-norm over $\FF_p$ for $\theta: \FF_p \rightarrow \CC$ as
\[
\|\theta\|_{u^s(\FF_p)}:= \sup_{\substack{P(y)\in \ZZ[y] \\ \text{deg}(P)\leq s-1}}
\left|
\EE_{y \in \FF_p} \theta(y) e_p(-P(y ))
\right|.
\]
And we say that $\theta$ is strongly $s$-uniform if there exist $C=C(\theta,s), c=c(\theta,s)>0$ which are independent of $p$ such that
\[
\|\theta-\EE \theta\|_{u^s(\FF_p)} \leq Cp^{-c}.
\]
\end{definition}

Consider the following weighted multilinear polynomial averaging operators:
\begin{equation} \label{weighted multilinear polynomial averaging operators}
    G^{\theta}_{k}(\ux):= \EE_{y \in \FF_p} \theta(y) \prod_{i=1}^k f_i(\ux+P_i(y)v_i).
\end{equation}
The main result of this paper is the following inequality for the weighted multilinear polynomial averaging operators in finite fields:

\begin{theorem} \label{thm_weighted polynomial multiple ergodic averages}
    Let $\theta: \FF_p \rightarrow \CC$ be 1-bounded. Assume that $\theta$ is strongly $(d+1)$-uniform.
    Then there exist $p_0=p_0(\mathcal{V}, \cP,\theta) \in \NN$, and $c=c(\cP,\theta)>0$ such that for any primes $p > p_0$, we have
    \[
     \EE_{\ux \in \FF_p^D}
     \left|
    G^{\theta}_{k}(\ux)- \EE_{y \in \FF_p} \theta(y) \cdot
      \prod_{i=1}^k 
     \EE_{n_i\in \FF_p}f_i( \underline{x}+ n_iv_i)
     \right|^2
     =O_{\cP, \theta}\left(p^{-c} \right),
    \]
    for any 1-bounded functions $f_0, \cdots, f_k: \FF_p^D \rightarrow \CC $.
\end{theorem}

The study of the $\cL^2$-norm convergence of \eqref{weighted multilinear polynomial averaging operators} has a long history in ergodic theory. The unweighted case of Theorem \ref{thm_weighted polynomial multiple ergodic averages} was independently solved by Host and Kra \cite{HK05}, and Leibman \cite{L05} when $D=1$, and was settled by Walsh \cite{W12} for arbitrary $D$ in a much general setting in ergodic theory. For the weighted case, a prominent example is the prime-weighted ergodic averages, and the corresponding norm convergence problem is known due to Frantzikinakis, Host, and Kra \cite{FHK07} and Wooley and Ziegler \cite{WZ12}. Another example is the study of the Möbius weight, which was proved by Frantzikinakis and Host \cite{FH17}. We refer the readers to Frantzikinakis and Host \cite{FH17} for more discussion on more general weighted ergodic averages.

While all the results mentioned above are of a more general nature, Theorem \ref{thm_weighted polynomial multiple ergodic averages} focuses on the quantitative aspects, and can be viewed as a finite version of the quantitative weighted $\cL^2$-norm convergence result of Frantzikinakis and Kra \cite{FK05}.
Moreover, we provide a criterion in Theorem \ref{thm_weighted polynomial multiple ergodic averages} for when quantitative $\cL^2$-norm convergence occurs, and we will present a new application below.

Theorem \ref{thm_weighted polynomial multiple ergodic averages} will be deduced as a consequence of Proposition \ref{prop_u^d control} below, which concerns the $u^s$-norm control of the weighted averaging operators.
This formulation is motivated by Teräväinen \cite{T24}, in which he demonstrated that obtaining $u^s$-norm control rather than Gowers $U^s$-norm control is possible for particular polynomial multiple ergodic averages problems. Therefore, Proposition \ref{prop_u^d control} when $D=1$ is essential due to Teräväinen \cite{T24}, and we further explore this idea in the multidimensional case in finite fields.

\subsection{Weighted multidimensional polynomial Szemerédi theorem}

Consider the following weighted counting operators:
\[
\Lambda^{\theta}_{k}(f_0,...,f_k):=  \EE_{\underline{x}\in \FF^D_p} \EE_{y\in \FF_p} \theta(y) f_0(\underline{x}) \prod_{i=1}^k f_i( \underline{x}+ P_i(y)v_i ).
\]
As a consequence of Theorem \ref{thm_weighted polynomial multiple ergodic averages} by applying the Cauchy-Schwarz inequality, we have the following weighted multidimensional polynomial Szemerédi theorem in finite fields:

\begin{theorem} \label{thm_weighted counting polynomial progressions}
    Let $\theta: \FF_p \rightarrow \CC$ be 1-bounded. Assume that $\theta$ is strongly $(d+1)$-uniform.
    Then there exist $p_0=p_0(\mathcal{V}, \cP,\theta) \in \NN$, and $c=c(\cP,\theta)>0$ such that for any primes $p > p_0$, we have
    \[
     \Lambda^{\theta}_{k}(f_0,...,f_k)
     =\EE_{y \in \FF_p} \theta(y) \cdot 
     \EE_{\underline{x}\in \FF^D_p}  f_0(\underline{x}) \prod_{i=1}^k 
     \EE_{n_i\in \FF_p}f_i( \underline{x}+ n_iv_i)
    + O_{\cP, \theta}\left(p^{-c} \right),
    \]
    for any 1-bounded functions $f_0, \cdots, f_k: \FF_p^D \rightarrow \CC $.
\end{theorem}

The unweighted case of Theorem \ref{thm_weighted counting polynomial progressions} is known as the multidimensional polynomial Szemerédi theorem in finite fields. When $D=1$, Bourgain and Chang \cite{BC17} initiated the study of the following nonlinear Roth progression in finite fields:
\[
x,x+y,x+y^2.
\]
Dong, Li, and Sawin \cite{DLS20}, and Peluse \cite{P18} later independently extended this result to more general three-term polynomial progressions. After a while, longer-term polynomial progressions were solved by Peluse in her breakthrough paper \cite{P19}. To be more precise, Peluse studied the following polynomial progressions:
\[
x,x+P_1(y),\cdots,x+P_k(y),
\]
where $\{P_1, \cdots,P_k \in \ZZ[y\}$ is a collection of linearly independent polynomials with zero constant terms.

For the higher-dimensional case, Han, Lacey, and Yang \cite{HLY21} first extended to more general three-term polynomial progressions when $D=2$. Afterwards, Kuca generalized Peluse's polynomial Szemerédi theorem to the multi-dimensional setting in his papers \cite{K24b} and \cite{K24a}, and Theorem \ref{thm_weighted counting polynomial progressions} for $\theta =1$ is the main result in Kuca \cite{K24b}. Therefore, we view Theorem \ref{thm_weighted counting polynomial progressions} as a weighted multidimensional polynomial Szemerédi theorem in finite fields.

\subsection{Specific multidimensional rational function progressions}

In another direction, Bourgain and Chang \cite{BC17} asked if one can replace the polynomials by rational functions, and they consider the following configuration as evidence for the more general phenomenon of the "nonlinear" Szemerédi theorem:
\[
 x, x+y, x+1/y.
\]
The main distinction between the polynomials and the rational functions is that differentiation does not reduce the "complexity" of the rational functions, so that the framework in Peluse \cite{P19} does not apply. 

In \cite{HL25}, the author and Lim developed a new method, called algebraic geometry PET induction, to deal with more general three-term rational function progressions, and this was later extended to the higher-dimensional three-term case by Lim \cite{L25}. However, it appears that all the current techniques are limited to the three-term case, and significant new ideas are needed to study more general rational function progressions with longer terms.

The final result of this paper addresses this problem in a specific situation as a direct application of Theorem \ref{thm_weighted counting polynomial progressions} and Bombieri's bound for exponential sums in finite fields with rational function phases (see \cite[Proposition 2.2]{BC17}, for example).

\begin{proposition} \label{prop_counting rational function progressions}
Let $\varphi(y) \in \QQ(y)$ be a nonzero rational function. Then there exist $p_0=p_0(\varphi,\mathcal{V}, \cP) \in \NN$, and $c=c(\varphi,\cP)>0$ such that for any primes $p > p_0$, we have
    \[
    \EE_{\underline{x}\in \FF^D_p} \EE_{y\in \FF_p}^*  f_0(\underline{x}) \prod_{i=1}^k f_i( \underline{x}+ P_i(\varphi(y))v_i )
    \]
    \[
    =\EE_{\underline{x}\in \FF^D_p}  f_0(\underline{x}) \prod_{i=1}^k 
     \EE_{n_i\in \FF_p}f_i( \underline{x}+ n_iv_i)
    + O_{\cP,\varphi}\left(p^{-c} \right),
    \]
    for any 1-bounded functions $f_0, \cdots, f_k: \FF_p^D \rightarrow \CC $, where $\EE^*$ denotes the average excluding the poles of the rational function.
\end{proposition}

Proposition \ref{prop_counting rational function progressions} has the following corollary by applying \cite[Lemma 6.1]{K24a}.
\begin{corollary}
Let $\varphi(y) \in \QQ(y)$ be a nonzero rational function. Then there exist $p_0=p_0(\varphi,\mathcal{V}, \cP) \in \NN$, and $C=C(\varphi,\cP), c=c(\varphi,\cP)>0$ such that for any primes $p > p_0$, any $A \subset \FF^D_p$ with $|A|\geq Cp^{D-c}$ must contain a nontrivial configuration of the form
\[
\ux, \ux+ P_1(\varphi(y))v_1,\cdots, \ux+ P_k(\varphi(y))v_k.
\]
\end{corollary}

\subsection{Organization of the paper}
In Section \ref{Notation and preliminaries}, we provide some frequently used notation and establish the necessary background for the Gowers box norm and the relevant inverse theorem that will be used later. In Section \ref{Proof of thm_weighted polynomial multiple ergodic averages}, we give the proof of Theorem \ref{thm_weighted polynomial multiple ergodic averages}.

\subsection{Acknowledgements}
The author would like to thank his advisor, David Conlon, for his
encouragement and support.
The MOE Taiwan-Caltech Fellowship supported the author during the conduct of this research.

\section{Notation and preliminaries} \label{Notation and preliminaries}
\subsection{Notation and conventions}
Throughout the paper, we denote $p$ to be a sufficiently large prime. We use the Landau notation. We write $A=O(B)$ if there exists an absolute constant $C>0$ such that $|A| \leq C B$.
If the constant depends on other parameter, say $C=C(\mathcal{P})$, then we write $A=O_{\mathcal{P}}(B)$ to mean that $|A| \leq C(\mathcal{P}) B$.

We use $\NN$, $\ZZ$, $\QQ$, and $\FF_p$ to denote the set of positive integers, the set of integers, the set of rational numbers, and the finite field with $p$ elements, respectively.
For $x \in \FF_p$, we set
\[
e_p(x):= e^{2 \pi i x/p}.
\]
For any finite set $X$ and any function $f: X \rightarrow \CC$, we denote the average of $f$ over $X$ as
\[
\EE_{x \in X} f(x) := \frac{1}{|X|} \sum_{x\in X} f(x).
\]
Define the $\cL^p$-norm for $p \in (1,\infty)$ over $X$ as 
\[
\|f\|_{\cL^p(X)}:= \left( \EE_{x \in X} |f(x)|^p \right)^{1/p}.
\]
We say a function $f$ is 1-bounded if
\[
\sup_{x \in X} |f(x)| \leq 1.
\]
Let $f: \FF_p \rightarrow \CC$ be any function.
For $\ux \in \FF_p^D$, $v\in \ZZ^D$, and $\xi \in \FF_p$, we define the Fourier coefficient of $f$ along the direction $v$ as
\[
\widehat{f}(\ux;v;\xi) :=  \EE_{n \in \FF_p} f(\ux+nv) e_p(-n \xi),
\]
then we have the following Fourier inversion formula:
\[
f(\ux+nv)= \sum_{\xi \in \FF_p} \widehat{f}(\ux;v;\xi) e_p(n \xi),
\]
with Parseval's identity:
\[
\EE_{n \in \FF_p} \left| f(\ux+nv) \right|^2=
\sum_{\xi \in \FF_p} \left| \widehat{f}(\ux;v;\xi)  \right|^2.
\]
We record a simple observation about the size of the Fourier coefficients below that will be used later.
\begin{lemma} \label{lma_fourier coefficients}
    Let $v \in \ZZ^D$ be any nonzero vector. Then for any $\xi \in \FF_p$, we have
    \[
    \left|  \widehat{f}(\ux;v;\xi) \right| = \left|  \widehat{f}(\ux';v;\xi) \right|,
    \]
    if $\ux-\ux' \in \langle v \rangle$, where $ \langle v \rangle$ denotes the linear span of $v$ over $\FF_p$.
\end{lemma}

\subsection{Gowers box norm along vector subspaces}
Now let $D,s \in \NN$, and consider $\FF_p^D$ as the ambient space. For $\ux,\uh \in \FF_p^D$, define the multiplicative derivative
\[
\triangle_{\uh}f(\ux):=f(\ux)\overline{f}(\ux+\uh).
\]
For $\uh_1, \cdots, \uh_s \in \FF_p^D$, we write
\[
\triangle_{\uh_1, \cdots, \uh_s}f(\ux):=\triangle_{\uh_1} \cdots \triangle_{\uh_s}f(\ux).
\]
Given subspaces $H_1, \cdots, H_s \subset \FF_p^D$, we define the Gowers box norm of $f: \FF_p^D \rightarrow \CC$ along $H_1, \cdots, H_s$ as 
\begin{equation} \label{definition of Box norm}
    \|f\|_{H_1, \cdots, H_s}:= 
\left(
\EE_{\ux \in \FF_p^D} \EE_{\uh_1 \in H_1} \cdots \EE_{\uh_s \in H_s}
\triangle_{\uh_1, \cdots, \uh_s}f(\ux)
\right)^{1/2^s}.
\end{equation}
When $D=1$, and $H_1=\cdots=H_s=\FF_p$, \eqref{definition of Box norm} becomes the usual Gowers $U^s$-norm, and is denoted by $\|f\|_{U^s(\FF_p)}$.
For more properties of the Gowers box norm, we refer the readers to \cite[Section 2]{K24b}.

In this paper, we consider only the special case where $s=2$, and we record the inverse theorem for the following specific box norm:

\begin{lemma} \label{lma_U^2 inverse theorem}
    Let $f: \FF_p^D \rightarrow \CC $ be 1-bounded. Then for any nonzero vector $v \in \ZZ^D$, we have
    \[
    \|f\|^4_{\FF_p^D,\langle v \rangle} \leq \sup_{\xi \in \FF_p} \EE_{\ux \in \FF_p^D} \left| \widehat{f}(\ux;v;\xi) \right|^2.
    \]
\end{lemma}

\begin{proof} [Proof of Lemma \ref{lma_U^2 inverse theorem}]
    For $\ux, \uh \in \FF_p^D$, we write
    \[
    \ux=\ux'+x_v v \in \langle v \rangle^{\perp} \oplus \langle v \rangle, \quad \text{and} \quad \uh=\uh'+h_v v \in \langle v \rangle^{\perp} \oplus \langle v \rangle,
    \]
    for some $x_v, h_v \in \FF_p$.
    Then we have
    \[
    \|f\|^4_{\FF_p^D,\langle v \rangle}= \EE_{\ux, \uh \in \FF_p^D} \EE_{n\in\FF_p} \triangle_{\uh, nv}f(\ux)
    \]
    \[
    = \EE_{\ux', \uh' \in \langle v \rangle^{\perp}} \EE_{x_v, h_v, n \in \FF_p} f(\ux'+x_v v) \overline{f}(\ux'+\uh'+(x_v+h_v)v) 
    \]
    \[
    \times \overline{f}(\ux'+(x_v+n)v) f(\ux'+\uh'+(x_v+h_v+n)v)
    \]
    \[
    = \EE_{\ux', \uh' \in \langle v \rangle^{\perp}}  \sum_{\xi \in \FF_p}
    \left| \widehat{f}(\ux';v;\xi)  \right|^2 
    \left| \widehat{f}(\ux'+\uh';v;\xi)  \right|^2 
    \]
    \[
    =\sum_{\xi \in \FF_p} \left(  \EE_{\ux' \in \langle v \rangle^{\perp}} \left| \widehat{f}(\ux';v;\xi)  \right|^2 \right)^2
    \]
    \[
    \leq \left( \sum_{\xi \in \FF_p} \EE_{\ux' \in \langle v \rangle^{\perp}} \left| \widehat{f}(\ux';v;\xi)  \right|^2 \right)
    \left( \sup_{\xi \in \FF_p}  \EE_{\ux' \in \langle v \rangle^{\perp}} \left| \widehat{f}(\ux';v;\xi)  \right|^2 \right).
    \]
    For the first term, by Parseval's identity, we have
    \[
    \sum_{\xi \in \FF_p} \EE_{\ux' \in \langle v \rangle^{\perp}} \left| \widehat{f}(\ux';v;\xi)  \right|^2 =  \EE_{\ux' \in \langle v \rangle^{\perp}} \sum_{\xi \in \FF_p} \left| \widehat{f}(\ux';v;\xi)  \right|^2
    \]
    \[
    = \EE_{\ux' \in \langle v \rangle^{\perp}} \EE_{n \in \FF_p}
    \left| f(\ux'+nv)  \right|^2 \leq 1.
    \]
    For the second term, by Lemma \ref{lma_fourier coefficients}, we have
    \[
    \EE_{\ux' \in \langle v \rangle^{\perp}} \left| \widehat{f}(\ux';v;\xi)  \right|^2 = \EE_{\ux' \in \langle v \rangle^{\perp}} \left| \widehat{f}(\ux'+nv;v;\xi)  \right|^2,
    \]
    for any $n \in \FF_p$. Therefore, we have
    \[
    \EE_{\ux' \in \langle v \rangle^{\perp}} \left| \widehat{f}(\ux';v;\xi)  \right|^2 = \EE_{n \in \FF_p} \EE_{\ux' \in \langle v \rangle^{\perp}} \left| \widehat{f}(\ux'+nv;v;\xi)  \right|^2 
    \]
    \[
    =\EE_{\ux \in \FF_p^D} \left| \widehat{f}(\ux;v;\xi)  \right|^2,
    \]
    and the proof is complete.
\end{proof}

\section{Proof of Theorem \ref{thm_weighted polynomial multiple ergodic averages}} \label{Proof of thm_weighted polynomial multiple ergodic averages}
To avoid further confusion, we state the unweighted case of Theorem \ref{thm_weighted counting polynomial progressions} here for the reader's convenience, and we will use this result as a black box.
\begin{theorem} \label{thm_counting polynomial progressions}(\cite{K24b})
    There exist $p_0=p_0(\mathcal{V}, \cP) \in \NN$, and $c=c(\cP)>0$ such that for any primes $p > p_0$, we have
    \[
     \Lambda^1_{k}(f_0,...,f_k)
     = \EE_{\underline{x}\in \FF^D_p}  f_0(\underline{x}) \prod_{i=1}^k 
     \EE_{n_i\in \FF_p}f_i( \underline{x}+ n_iv_i)
    + O_{\cP}\left(p^{-c} \right),
    \]
    for any 1-bounded functions $f_0, \cdots, f_k: \FF_p^D \rightarrow \CC $.
\end{theorem}
Throughout the remainder of this section, let $l\in \NN$ with $l \leq k$.
Define
\[
G^{\theta}_{l,k}(\ux):= \EE_{y \in \FF_p} \theta(y) \prod_{i=1}^l f_i(\ux+P_i(y)v_i) \prod_{i=l+1}^k e_p(P_i(y)\xi_i),
\]
where $\xi_{l+1},\cdots, \xi_k \in \FF_p$. When $l=k$, we simply write $G^{\theta}_{k}=G^{\theta}_{l,k}$. Considering these more general weighted averaging operators plays a crucial role when we run an induction argument later. The most essential part of this paper is the following result, which concerns the control of the weighted averaging operators in terms of $u^s$-norm:
\begin{proposition} \label{prop_u^d control}
    There exist $p_0=p_0(\mathcal{V}, \cP) \in \NN$, and $c_l=c_l(\cP)>0$ such that for any primes $p > p_0$, we have
    \[
    \|G^{\theta}_{l,k} \|_{\cL^2(\FF_p^D)}^{O_{\cP}(1)}
    \leq \| \theta \|_{u^{d+1}(\FF_p)} + O_{\cP}\left(p^{-c_l} \right),
    \]
    for any 1-bounded functions $f_1, \cdots, f_l: \FF_p^D \rightarrow \CC $, any 1-bounded function $\theta: \FF_p \rightarrow \CC$, and any $\xi_{l+1},\cdots, \xi_k \in \FF_p$.
\end{proposition}

We postpone the proof of Proposition \ref{prop_u^d control} until the end of this section, and we first start with the proof of Theorem \ref{thm_weighted polynomial multiple ergodic averages}:

\begin{proof}[Proof of Theorem \ref{thm_weighted polynomial multiple ergodic averages} assuming Proposition \ref{prop_u^d control}]
If we set
\[
G^{\EE \theta}(\ux):=\EE_{y \in \FF_p} \theta(y) \cdot
      \prod_{i=1}^k 
     \EE_{n_i\in \FF_p}f_i( \underline{x}+ n_iv_i),
\]
then we have
\[
     \EE_{\ux \in \FF_p^D}
     \left|
    G^{\theta}_{k}(\ux)- \EE_{y \in \FF_p} \theta(y) \cdot
      \prod_{i=1}^k 
     \EE_{n_i\in \FF_p}f_i( \underline{x}+ n_iv_i)
     \right|^2
     =\EE_{\ux \in \FF_p^D}
     \left|
    G^{\theta}_{k}(\ux)- G^{\EE \theta}(\ux)
     \right|^2
\]
\begin{equation} \label{1_proof of weighted polynomial multiple ergodic averages}
    =\langle   G^{\theta}_{k},  G^{\theta}_{k}  \rangle
-\langle  G^{\theta}_{k} , G^{\EE \theta} \rangle
-\langle G^{\EE \theta} ,  G^{\theta}_{k} \rangle
+\langle G^{\EE \theta}, G^{\EE \theta} \rangle,
\end{equation}
where we denote $\langle f,g \rangle := \EE_{\ux \in \FF_p^D} f(\ux) \overline{g}(\ux)$.

Next, we write
\[
    G^{\theta}_{k}=G^{\theta-\EE \theta}_{k}+G^{\EE \theta}_{k}.
\]
Then by Proposition \ref{prop_u^d control} and the Cauchy-Schwarz inequality, there exists $c_1=c_1(\cP)>0$ such that
\[
\left|  \langle   G^{\theta-\EE \theta}_{k},  f  \rangle  \right| ^{O_{\cP}(1)}
=  O_{\cP}\left(\| \theta-\EE \theta \|_{u^{d+1}(\FF_p)}+p^{-c_1} \right),
\]
for any 1-bounded function $f$. And from the assumption that $\theta$ is strongly $(d+1)$-uniform, there exists $c_2=c_2(\cP,\theta)>0$ such that
\[
\|\theta-\EE \theta\|_{u^{d+1}(\FF_p)} =O_{\theta}\left( p^{-c_2} \right).
\]
To conclude, we have shown that 
\[
\langle   G^{\theta}_{k},  G^{\theta}_{k}  \rangle
-\langle  G^{\theta}_{k} , G^{\EE \theta} \rangle
-\langle G^{\EE \theta} ,  G^{\theta}_{k} \rangle
+\langle G^{\EE \theta}, G^{\EE \theta} \rangle
\]
\begin{equation} \label{2_proof of weighted polynomial multiple ergodic averages}
    =\langle   G^{\EE \theta}_{k},  G^{\EE \theta}_{k}  \rangle
-\langle  G^{\EE \theta}_{k} , G^{\EE \theta} \rangle
-\langle G^{\EE \theta} ,  G^{\EE \theta}_{k} \rangle
+\langle G^{\EE \theta}, G^{\EE \theta} \rangle
+O_{\cP,\theta}\left( p^{-c_3} \right),
\end{equation}
for some $c_3=c_3(\cP,\theta)$.

Finally, by Theorem \ref{thm_counting polynomial progressions}, there exists 
$c_4=c_4(\cP,\theta)>0$ such that
\[
 \langle G^{\EE \theta}_{k},  f \rangle= \langle G^{\EE \theta},  f \rangle
 +O_{\cP, \theta}\left( p^{-c_4} \right),
\]
for any 1-bounded function $f$. Therefore, we have
\[
\langle   G^{\EE \theta}_{k},  G^{\EE \theta}_{k}  \rangle
-\langle  G^{\EE \theta}_{k} , G^{\EE \theta} \rangle
-\langle G^{\EE \theta} ,  G^{\EE \theta}_{k} \rangle
+\langle G^{\EE \theta}, G^{\EE \theta} \rangle
\]
\[
=\langle   G^{\EE \theta},  G^{\EE \theta}  \rangle
-\langle  G^{\EE \theta} , G^{\EE \theta} \rangle
-\langle G^{\EE \theta} ,  G^{\EE \theta} \rangle
+\langle G^{\EE \theta}, G^{\EE \theta} \rangle
+O_{\cP, \theta}\left( p^{-c_4} \right)
\]
\begin{equation} \label{3_proof of weighted polynomial multiple ergodic averages}
    =O_{\cP, \theta}\left( p^{-c_4} \right).
\end{equation}
Combining \eqref{1_proof of weighted polynomial multiple ergodic averages}, \eqref{2_proof of weighted polynomial multiple ergodic averages}, and \eqref{3_proof of weighted polynomial multiple ergodic averages} concludes the proof.

\end{proof}

It remains to prove Proposition \ref{prop_u^d control}. The strategy of the proof consists of two steps: a suitable box norm control and a degree-lowering procedure for the relevant box norm.
This framework dates back to Peluse \cite{P19} in the one-dimensional case, and was later generalized to the higher-dimensional case by Kuca \cite{K24b} and \cite{K24a}. The details differ in that different box norm controls require different degree-lowering procedures, and we extend this framework to our current setting.

Before stating the proof, we first define the following version of the "dual functions," which are the central objects on which we will perform the degree-lowering procedure later:
\[
F^{\theta}_{l,k}(\ux):= \EE_{y,y' \in \FF_p} \theta(y)\overline{\theta}(y')
\left(  \prod_{i=1}^{l-1} f_i(\ux+P_i(y)v_i-P_l(y')v_l) \overline{f_i}(\ux+P_i(y')v_i-P_l(y')v_l) \right)
\]
\[
\times f_l(\ux+P_l(y)v_l-P_l(y')v_l) 
\left(  \prod_{i=l+1}^k e_p\left((P_i(y)-P_i(y')) \xi_i \right) \right).
\]

The following lemma concerns the suitable box norm controls of the weighted averaging operators in terms of the dual functions, which is the first step of our analysis:
\begin{lemma} \label{lma_Gowers norm control}
    There exist $p_0=p_0(\mathcal{V}, \cP) \in \NN$, and $c_l=c_l(\cP)>0$ such that for any primes $p > p_0$, we have
    \[
    \|G^{\theta}_{l,k} \|_{\cL^2(\FF_p^D)}^{O(1)} \leq 
    \|F^{\theta}_{l,k}\|_{\FF_p^D,\langle v_l \rangle} +O(p^{-c_l}),
    \]
    for any 1-bounded functions $f_1, \cdots, f_l: \FF_p^D \rightarrow \CC $, any 1-bounded function $\theta: \FF_p \rightarrow \CC$, and any $\xi_{l+1},\cdots, \xi_k \in \FF_p$.
\end{lemma}

\begin{proof} [Proof of Lemma \ref{lma_Gowers norm control}]
    From the definition of $G^{\theta}_{l,k}$ and $F^{\theta}_{l,k}$, we have
    \[
    \|G^{\theta}_{l,k} \|_{\cL^2(\FF_p^D)}^2= \EE_{\ux \in \FF_p^D} F^{\theta}_{l,k}(\ux) \overline{f_l} (\ux).
    \]
    Applying the Cauchy-Schwarz inequality to the identity above, we have
    \[
    \|G^{\theta}_{l,k} \|_{\cL^2(\FF_p^D)}^4 \leq \|F^{\theta}_{l,k} \|_{\cL^2(\FF_p^D)}^2
    \]
    \[
= \EE_{\ux \in \FF_p^D}\EE_{y,y' \in \FF_p} \theta(y)\overline{\theta}(y')
\left(  \prod_{i=1}^{l-1} f_i(\ux+P_i(y)v_i-P_l(y')v_l) \overline{f_i}(\ux+P_i(y')v_i-P_l(y')v_l) \right)
\]
\[
\times f_l(\ux+P_l(y)v_l-P_l(y')v_l) \overline{F^{\theta}_{l,k}}(\ux)
\left(  \prod_{i=l+1}^k e_p\left((P_i(y)-P_i(y')) \xi_i \right) \right)
\]
\[
= \EE_{\ux \in \FF_p^D}\EE_{y,y' \in \FF_p} \theta(y)\overline{\theta}(y')
\left(  \prod_{i=1}^{l-1} f_i(\ux+P_i(y)v_i) \overline{f_i}(\ux+P_i(y')v_i) \right)
\]
\[
\times f_l(\ux+P_l(y)v_l) \overline{F^{\theta}_{l,k}}(\ux+P_l(y')v_l)
\left(  \prod_{i=l+1}^k e_p\left((P_i(y)-P_i(y')) \xi_i \right) \right),
\]
where we make the change of variable $\ux \leftrightarrow \ux +P_l(y')v_l$ in the previous step. Apply the Cauchy-Schwarz inequality to the $y,y'$ variables, we have
\[
\|G^{\theta}_{l,k} \|_{\cL^2(\FF_p^D)}^8 \leq \EE_{y,y' \in \FF_p}
\]
\[
\times
\left| \EE_{\ux \in \FF_p^D}
\left(  \prod_{i=1}^{l-1} f_i(\ux+P_i(y)v_i) \overline{f_i}(\ux+P_i(y')v_i) \right)  f_l(\ux+P_l(y)v_l) \overline{F^{\theta}_{l,k}}(\ux+P_l(y')v_l)
\right|^2
\]
\[
= \EE_{y,y' \in \FF_p} \EE_{\ux, \uh \in \FF_p^D}
\left(  \prod_{i=1}^{l-1} \triangle_{\uh} f_i(\ux+P_i(y)v_i) \triangle_{\uh}\overline{f_i}(\ux+P_i(y')v_i) \right)  
\]
\[
\times \triangle_{\uh}f_l(\ux+P_l(y)v_l) \triangle_{\uh}\overline{F^{\theta}_{l,k}}(\ux+P_l(y')v_l)
\]
\[
=\EE_{\ux, \uh \in \FF_p^D} 
\left(
\EE_{y \in \FF_p}
\prod_{i=1}^{l} \triangle_{\uh} f_i(\ux+P_i(y)v_i) 
\right)
 \]
 \begin{equation} \label{1_proof_Gowers norm control}
 \times
     \left(
\EE_{y' \in \FF_p} \triangle_{\uh}\overline{F^{\theta}_{l,k}}(\ux+P_l(y')v_l)
\prod_{i=1}^{l-1} \triangle_{\uh}\overline{f_i}(\ux+P_i(y')v_i)
\right).
 \end{equation}
If we set 
\[
F(\ux):=\EE_{y' \in \FF_p} \triangle_{\uh}\overline{F^{\theta}_{l,k}}(\ux+P_l(y')v_l)
\prod_{i=1}^{l-1} \triangle_{\uh}\overline{f_i}(\ux+P_i(y')v_i),
\]
then by Theorem \ref{thm_counting polynomial progressions}, there exists $c'_l=c'_l(\cP)>0$ such that
\[
 \EE_{\ux \in \FF_p^D} F(\ux) 
\EE_{y \in \FF_p}
\prod_{i=1}^{l} \triangle_{\uh} f_i(\ux+P_i(y)v_i) 
\]
\[
= \EE_{\ux \in \FF_p^D} F(\ux) \prod_{i=1}^{l} 
\EE_{n_i \in \FF_p} \triangle_{\uh} f_i(\ux+n_iv_i) 
+O_{\cP}(p^{-c_{l}'}).
\]
Next, if we set
\[
F'(\ux):=\prod_{i=1}^{l} 
\EE_{n_i \in \FF_p} \triangle_{\uh} f_i(\ux+n_iv_i),
\]
then by Theorem \ref{thm_counting polynomial progressions} again, we have
\[
\EE_{\ux \in \FF_p^D} F(\ux) \prod_{i=1}^{l} 
\EE_{n_i \in \FF_p} \triangle_{\uh} f_i(\ux+n_iv_i) 
\]
\[
=\EE_{\ux \in \FF_p^D} F'(\ux) \EE_{y' \in \FF_p} \triangle_{\uh}\overline{F^{\theta}_{l,k}}(\ux+P_l(y')v_l)
\prod_{i=1}^{l-1} \triangle_{\uh}\overline{f_i}(\ux+P_i(y')v_i)
\]
\[
= \EE_{\ux \in \FF_p^D} F'(\ux) 
\EE_{n_l \in \FF_p} \triangle_{\uh}\overline{F^{\theta}_{l,k}}(\ux+n_lv_l)
\prod_{i=1}^{l-1} 
\EE_{n_i \in \FF_p} \triangle_{\uh} \overline{f_i}(\ux+n_iv_i) 
+O_{\cP}(p^{-c_{l}'}).
\]
To conclude, we have shown that
\[
\EE_{\ux \in \FF_p^D} 
\left(
\EE_{y \in \FF_p}
\prod_{i=1}^{l} \triangle_{\uh} f_i(\ux+P_i(y)v_i) 
\right)
\]
\[
\times \left(
\EE_{y' \in \FF_p} \triangle_{\uh}\overline{F^{\theta}_{l,k}}(\ux+P_l(y')v_l)
\prod_{i=1}^{l-1} \triangle_{\uh}\overline{f_i}(\ux+P_i(y')v_i)
\right)
\]
\begin{equation} \label{2_proof_Gowers norm control}
    =\EE_{\ux \in \FF_p^D}  F''(\ux) \EE_{n_l \in \FF_p} \triangle_{\uh}\overline{F^{\theta}_{l,k}}(\ux+n_lv_l)+O_{\cP}(p^{-c_{l}'}),
\end{equation}
for some function 1-bounded $F''(\ux)$. Combining \eqref{1_proof_Gowers norm control} and \eqref{2_proof_Gowers norm control}, we have
\[
\|G^{\theta}_{l,k} \|_{\cL^2(\FF_p^D)}^8 \leq
\EE_{\ux,\uh \in \FF_p^D}  F''(\ux) \EE_{n_l \in \FF_p} \triangle_{\uh}\overline{F^{\theta}_{l,k}}(\ux+n_lv_l)+O_{\cP}(p^{-c_{l}'})
\]
\begin{equation} \label{3_proof_Gowers norm control}
    =\EE_{\uh \in \FF_p^D} \EE_{\ux' \in \langle v_l \rangle^{\perp}}
\left(
\EE_{m_l \in \FF_p} F''(\ux'+m_lv_l)
\right)
\left(
\EE_{n_l \in \FF_p} \triangle_{\uh}\overline{F^{\theta}_{l,k}}(\ux'+n_lv_l)
\right)+O_{\cP}(p^{-c_{l}'}),
\end{equation}
where we write
    \[
    \ux=\ux'+m_l v_l \in \langle v_l \rangle^{\perp} \oplus \langle v_l \rangle, 
    \]
    and make the change of variable $n_l \leftrightarrow n_l-m_l$ in the previous step. Finally, apply the Cauchy-Schwarz inequality to \eqref{3_proof_Gowers norm control}, we have
    \[
    \|G^{\theta}_{l,k} \|_{\cL^2(\FF_p^D)}^{16} \leq
    \EE_{\uh \in \FF_p^D} \EE_{\ux' \in \langle v_l \rangle^{\perp}}
    \left|
    \EE_{n_l \in \FF_p} \triangle_{\uh}\overline{F^{\theta}_{l,k}}(\ux'+n_lv_l)
    \right|^2
    +O_{\cP}(p^{-c_{l}'})
    \]
    \[
    = \EE_{\uh \in \FF_p^D} \EE_{\ux' \in \langle v_l \rangle^{\perp}}
    \EE_{n_l,m_l \in \FF_p} \triangle_{\uh, n_lv_l} F^{\theta}_{l,k}(\ux'+m_lv_l)+O_{\cP}(p^{-c_{l}'})
    \]
    \[
    = \EE_{\uh \in \FF_p^D} \EE_{\ux \in \FF_p^D}
    \EE_{n_l \in \FF_p} \triangle_{\uh, n_lv_l} F^{\theta}_{l,k}(\ux)+O_{\cP}(p^{-c_{l}'})
    \]
    \[
    =\|F^{\theta}_{l,k}\|_{\FF_p^D,\langle v_l \rangle}^4 +O(p^{-c_l}),
    \]
    and the proof is complete.
\end{proof}

Now we are ready to perform the second step of our analysis, the degree-lowering procedure on the dual functions, to finish the proof of Proposition \ref{prop_u^d control}.
\begin{proof} [Proof of Proposition \ref{prop_u^d control}]
    We will proceed by induction on the parameter $l$, while $k$ is a fixed integer. We begin with the base case, where $l=1$. Applying Parseval's identity, we have
    \[
    \|G^{\theta}_{1,k} \|_{\cL^2(\FF_p^D)}^{2}=
    \sum_{\underline{\eta} \in \FF^D_p} \left| \widehat{G^{\theta}_{1,k}}(\underline{\eta})  \right|^2
    \]
    \[
    =\sum_{\underline{\eta} \in \FF^D_p} \left|  \widehat{f_1}(\underline{\eta})
    \EE_{y \in \FF_p} \theta(y) e_p(P_1(y)(\underline{\eta}\cdot v_1))
    \prod_{i=2}^k e_p(P_i(y)\xi_i)
    \right|^2
    \]
    \[
    \leq \sum_{\underline{\eta} \in \FF^D_p} \left|  \widehat{f_1}(\underline{\eta}) \right|^2 \cdot \| \theta \|_{u^{d+1}(\FF_p)}^2
    \]
    \[
    = \EE_{\ux \in \FF^D_p} \left|  f_1(\ux) \right|^2 \cdot \| \theta \|_{u^{d+1}(\FF_p)}^2
    \]
    \[
    \leq \| \theta \|_{u^{d+1}(\FF_p)}^2,
    \]
    and this finishes the proof when $l=1$.
    
    From now on, we will suppose that $2 \leq l \leq k$, and that Proposition \ref{prop_u^d control} is true for $(l-1,k)$. By Lemma \ref{lma_Gowers norm control}, there exists $c_l'=c_l'(\cP)>0$ such that
    \begin{equation} \label{1_proof in prop3.1}
        \|G^{\theta}_{l,k} \|_{\cL^2(\FF_p^D)}^{O(1)} \leq 
    \|F^{\theta}_{l,k}\|_{\FF_p^D,\langle v_l \rangle} +O(p^{-c_l'}).
    \end{equation}
    Next, by the inverse theorem for the box norm, Lemma \ref{lma_U^2 inverse theorem}, we have
    \begin{equation} \label{2_proof in prop3.1}
        \|F^{\theta}_{l,k}\|^4_{\FF_p^D,\langle v_l \rangle} \leq \sup_{\xi_l \in \FF_p} \EE_{\ux \in \FF_p^D} \left| \widehat{F^{\theta}_{l,k}}(\ux;v_l;\xi_l) \right|^2.
    \end{equation}
    From the definition of $F^{\theta}_{l,k}$, we have
    \[
    \widehat{F^{\theta}_{l,k}}(\ux;v_l;\xi_l)=
    \EE_{n \in \FF_p} F^{\theta}_{l,k}(\ux+nv_l) e_p(-n\xi_l)
    \]
    \[
= \EE_{y,y',n \in \FF_p} \theta(y)\overline{\theta}(y')
\left(  \prod_{i=1}^{l-1} f_i(\ux+nv_l+P_i(y)v_i-P_l(y')v_l) \overline{f_i}(\ux+nv_l+P_i(y')v_i-P_l(y')v_l) \right)
\]
\[
\times f_l(\ux+nv_l+P_l(y)v_l-P_l(y')v_l) e_p(-n\xi_l)
\left(  \prod_{i=l+1}^k e_p\left( (P_i(y)-P_i(y'))\xi_i \right) \right)
\]
\[
= \EE_{y,y',n \in \FF_p} \theta(y)\overline{\theta}(y')
\left(  \prod_{i=1}^{l-1} f_i(\ux+nv_l+P_i(y)v_i) \overline{f_i}(\ux+nv_l+P_i(y')v_i) \right)
\]
\[
\times f_l(\ux+nv_l+P_l(y)v_l) e_p(-(n+P_l(y'))\xi_l)
\left(  \prod_{i=l+1}^k e_p\left( (P_i(y)-P_i(y'))\xi_i \right) \right),
\]
where we make the change of variable $n \leftrightarrow n+P_l(y')$ in the previous step. By organizing the terms above, we have
\[
= \EE_{y,n \in \FF_p} \theta(y) e_p(-n\xi_l) \prod_{i=1}^{l} f_i(\ux+nv_l+P_i(y)v_i)  \prod_{i=l+1}^k e_p\left( P_i(y) \xi_i \right) \
\]
\[
\times \EE_{y' \in \FF_p} \overline{\theta}(y') \prod_{i=1}^{l-1} \overline{f_i}(\ux+nv_l+P_i(y')v_i)  \prod_{i=l}^k e_p\left( -P_i(y')\xi_i \right).
\]
By applying the Cauchy-Schwarz inequality, we have
\[
\left| \widehat{F^{\theta}_{l,k}}(\ux;v_l;\xi_l) \right|^2 \leq
    \EE_{n \in \FF_p}
    \left| \EE_{y' \in \FF_p} \overline{\theta}(y') \prod_{i=1}^{l-1} \overline{f_i}(\ux+nv_l+P_i(y')v_i)  \prod_{i=l}^k e_p\left( -P_i(y')\xi_i \right) \right|^2.
\]
Therefore, by the change of variable $\ux \leftrightarrow \ux-nv_l$, we have
\[
\EE_{\ux \in \FF_p^D} \left| \widehat{F^{\theta}_{l,k}}(\ux;v_l;\xi_l) \right|^2
\]
\[
 \leq
\EE_{\ux \in \FF_p^D}
\EE_{n \in \FF_p}
    \left| \EE_{y' \in \FF_p} \overline{\theta}(y') \prod_{i=1}^{l-1} \overline{f_i}(\ux+nv_l+P_i(y')v_i)  \prod_{i=l}^k e_p\left( -P_i(y')\xi_i \right) \right|^2
\]
\[
=
    \EE_{\ux \in \FF_p^D} \left| \EE_{y' \in \FF_p} \overline{\theta}(y') \prod_{i=1}^{l-1} \overline{f_i}(\ux+P_i(y')v_i)  \prod_{i=l}^k e_p\left( -P_i(y')\xi_i \right) \right|^2
\]
\[
= \|G^{\theta}_{l-1,k} \|_{\cL^2(\FF_p^D)}^{2}.
\]
And hence, we apply the induction hypothesis to conclude that
\begin{equation} \label{3_proof in prop3.1}
    \EE_{\ux \in \FF_p^D} \left| \widehat{F^{\theta}_{l,k}}(\ux;v_l;\xi_l) \right|^{O_{\cP}(1)} 
    \leq \| \theta \|_{u^{d+1}(\FF_p)} + O_{\cP}\left(p^{-c_{l-1}} \right),
\end{equation}
for some $c_{l-1}=c_{l-1}(\cP)>0$.

Finally, combining \eqref{1_proof in prop3.1}, \eqref{2_proof in prop3.1}, and \eqref{3_proof in prop3.1} closes the induction, and the proof is complete.

\end{proof}

\bibliographystyle{alpha}
\bibliography{bibliography}

\end{document}